\documentclass[a4paper,11pt,twoside,notitlepage,reqno]{amsart}

\usepackage{amsmath,amssymb,amsbsy,amsthm}
\usepackage[hmargin=1.4in,vmargin=1.4in]{geometry}
\usepackage{url}
\numberwithin{equation}{section}
\usepackage{tikz}
\usetikzlibrary{matrix}
\usepackage[linktocpage]{hyperref}
\hypersetup{
    colorlinks=true,        
    linkcolor=red,          
    citecolor=red,         
    filecolor=magenta,      
    urlcolor=cyan           
}

\addtocontents{toc}{\protect\setcounter{tocdepth}{1}}

\usepackage{eucal}

\renewcommand{\P}{ \mathbb{P} }

\newcommand{\Z}{ \mathbb{Z} }
\newcommand{\N}{ \mathbb{N} }

\newcommand{\A}{ \mathbb{A} }

\theoremstyle{plain}
	\newtheorem{theorem}{Theorem}
	
		\numberwithin{theorem}{section}
	
	\newtheorem{proposition}[theorem]{Proposition}
	\newtheorem{corollary}[theorem]{Corollary}
	\newtheorem{conjecture}[theorem]{Conjecture}

	\newtheorem*{theorem*}{Theorem}
	\newtheorem*{lemma*}{Lemma}
	\newtheorem*{prop*}{Proposition}
	\newtheorem*{cor*}{Corollary}
	\newtheorem*{conj*}{Conjecture}

\theoremstyle{definition}
	\newtheorem{example}[theorem]{Example}
	\newtheorem*{example*}{Example}
	\newtheorem{definition}[theorem]{Definition}
	\newtheorem{remark}[theorem]{Remark}

\title[The Zariski dense orbit conjecture]{A conjecture strengthening the Zariski dense orbit problem for birational maps of dynamical degree one}

\author{Jason Bell}
\address{University of Waterloo \\
Department of Pure Mathematics \\
Waterloo, Ontario \\
N2L 3G1, Canada}
\email{jpbell@uwaterloo.ca}

\author{Dragos Ghioca}
\address{University of British Columbia\\
Department of Mathematics\\
Vancouver, BC\\
V6T 1Z2, Canada}
\email{dghioca@math.ubc.ca}

\thanks{The authors were partially supported by Discovery Grants from the National Science and Engineering Research Council of Canada.}


\begin{document}

\begin{abstract}
We formulate a strengthening of the Zariski dense orbit conjecture for birational maps of dynamical degree one. So, given a quasiprojective variety $X$ defined over an algebraically closed field $K$ of characteristic $0$, endowed with a birational self-map $\phi$ of dynamical degree $1$, we expect that  either there exists a non-constant rational function $f:X\dashrightarrow \P^1$ such that $f\circ \phi=f$, or there exists a proper subvariety $Y\subset X$ with the property that for any invariant proper subvariety $Z\subset X$, we have that $Z\subseteq Y$. We prove our conjecture for automorphisms $\phi$ of dynamical degree $1$ of semiabelian varieties $X$. Also, we prove a related result for regular dominant self-maps $\phi$ of semiabelian varieties $X$: assuming $\phi$ does not preserve a non-constant rational function, we have that the dynamical degree of $\phi$ is larger than $1$ if and only if the union of all $\phi$-invariant proper subvarieties of $X$ is Zariski dense. We give applications of our results to representation theoretic questions
about twisted homogeneous coordinate rings associated to abelian varieties.
\end{abstract}

\maketitle


\section{Introduction}
\label{sec:intro}


\subsection{The Zariski dense orbit conjecture}

The following conjecture was advanced by Medvedev-Scanlon \cite{M-S} and Amerik-Campana \cite{A-C} and was originally inspired by a conjecture of Zhang \cite{Zhang}.

\begin{conjecture}
\label{conj:original}
Let $X$ be a quasiprojective variety defined over an algebraically closed field $K$ of characteristic $0$, endowed with a rational dominant self-map $\phi$. Then exactly one of the following two conditions must hold:
\begin{enumerate}
\item[(I)] either there exists a non-constant rational function $f:X\dashrightarrow \P^1$ such that $f\circ \phi = f$ (which is referred to as $\phi$ preserving a non-constant rational function or fibration), 
\item[(II)] or there exists a point $x\in X(K)$ whose orbit under $\phi$ is well-defined (i.e., for each $n\ge 0$, the $n$-th iterate $\phi^n(x)$ lies outside the indeterminacy locus of $\phi$) and Zariski dense in $X$.
\end{enumerate}
\end{conjecture}

It is easy to see that conditions (I) and (II) are mutually exclusive; the difficulty lies in proving that in the absence of condition~(I), one can always find a point with a Zariski dense orbit as in~(II). Various cases of the above conjecture are known:
\begin{itemize}
\item when $K$ is uncountable (see \cite{A-C} and also \cite{BGZ}); 
\item when $X=\A^N$ and $\phi$ is the coordinatewise action of one-variable polynomials (see \cite{M-S});
\item when $X$ is a semiabelian variety and $\phi$ is a regular dominant self-map (see \cite{G-Matt} and also, see \cite{G-S, G-Sina} when $X$ is an abelian variety);
\item when $X$ is a commutative, linear algebraic group and $\phi$ is a group  endomorphism (see \cite{G-H} and also, see \cite{G-X});
\item when $X$ is a surface (see \cite{X}), and also for certain $3$-folds and higher dimensional varieties $X$ (see \cite{BGSZ}).
\end{itemize}

It is worth pointing out that when the field $K$ has positive characteristic, one would need to amend the statement of Conjecture~\ref{conj:original} (see \cite[Example~6.2]{BGZ} and especially, \cite[Conjecture~1.3]{G-Sina-20}).


\subsection{A strengthening of the conjecture in the case of birational maps of dynamical degree one}

We believe there is a stronger form of Conjecture~\ref{conj:original} when $\phi$ is a birational map of dynamical degree $1$. We recall that the \emph{dynamical degree} $\lambda_1(\phi)$ of a rational self-map $\phi$ of a projective smooth  variety $X$ of dimension $d$ is defined as
$$\lambda_1(\phi):=\lim_{n\to\infty} \left(\left(\phi^n\right)^*\mathcal{L}\cdot \mathcal{L}^{d-1}\right)^{\frac{1}{n}},$$
where $\mathcal{L}$ is an ample line bundle on $X$. The above limit exists and it is independent of the choice of the ample divisor $\mathcal{L}$;  
for more properties regarding the dynamical degree for self-maps of projective varieties along with its connections to the arithmetic degree, we refer the reader to \cite{Annali} and the references therein. 

Now, before stating our main conjecture, we need the following definition.
\begin{definition}
\label{def:invariant}
Let $\phi:X\dashrightarrow  X$ be a dominant rational self-map. A  subvariety $Y\subset X$ (not necessarily irreducible) is called invariant under $\phi$ if the restriction $\phi|_Y$ induces a dominant rational self-map of $Y$.
\end{definition}

\begin{conjecture}
\label{conj:strong}
Let $X$ be a smooth projective variety defined over an algebraically closed field $K$ of characteristic $0$ and let $\phi:X\dashrightarrow X$ be a birational self-map of dynamical degree $1$. Then exactly one of the following two statements must hold:
\begin{enumerate}
\item[(i)] either there exists a non-constant rational function $f:X\dashrightarrow \P^1$ such that $f\circ \phi=f$; 
\item[(ii)] or there exists a proper subvariety $Y\subset X$ with the property that each proper invariant subvariety $Z\subset X$ must be contained in $Y$.
\end{enumerate} 
\end{conjecture}

It is easy to see that condition~(ii) from Conjecture~\ref{conj:strong} implies the weaker condition~(II) from Conjecture~\ref{conj:original} (at least in the case $\phi$ is a regular morphism) since one can  choose a point $x\in (X\setminus Y)(K)$  and thus its orbit $\mathcal{O}_\phi(x)$ must be Zariski dense in $X$ because otherwise its Zariski closure would need to be contained in $Y$, which would be a contradiction. Also, we note below the following simple example which shows that oftentimes the subvariety $Y$ from condition~(ii) above is a nontrivial proper subvariety of $X$.

\begin{example}
\label{ex}
Let $\phi:\P^2\longrightarrow \P^2$ be given by $\phi\left([x:y:z]\right)=[2x:3y:z]$; 
then clearly there is no invariant fibration for $\phi$ since most points would have a Zariski dense orbit under $\phi$ (for example, the orbit of $[1:1:1]$ consists of all points of the form $\left[2^n:3^n:1\right]$, for $n\ge 0$). However, there exists a (nontrivial) proper subvariety $Y\subset \P^2$ containing all the proper $\phi$-invariant subvarieties of $X$; indeed, $Y$ consists of $3$ lines, as it is given by the equation $xyz=0$ in $\P^2$.
\end{example}

Furthermore, we believe an even stronger statement would hold, as follows.
\begin{conjecture}
\label{conj:iff}
Let $X$ be a smooth projective variety defined over an algebraically closed field $K$ of characteristic $0$ and let $\phi:X\dashrightarrow X$ be a dominant rational self-map. Assume there exists no non-constant rational function $f:X\dashrightarrow \P^1$ such that $f\circ \phi=f$. Then exactly one of the following two statements must hold:
\begin{itemize}
\item[(I)] either the dynamical degree of $\phi$ equals $1$,  
\item[(II)] or the union of all $\phi$-invariant proper subvarieties of $X$ is Zariski dense.
\end{itemize}
\end{conjecture}


\subsection{Motivation for our conjectures}


It is possible that neither item~(i) nor (ii) in Conjecture \ref{conj:strong} holds if one does not impose the constraint on the dynamical degree. For example, every automorphism of $\mathbb{A}^2$ of dynamical degree greater than one has a Zariski dense set of periodic points and does not preserve a non-constant fibration (see Xie \cite[Theorem 1.1(i)]{Xie1}).  On the other hand, work of Cantat \cite{Can}, Diller and Favre \cite{DF}, along with work of Xie \cite[Theorem 1.1]{Xie1} shows that for birational maps $\phi$ of surfaces over algebraically closed base fields of characteristic zero, exactly one of (i) and (ii) in Conjecture \ref{conj:iff} must hold when $\phi$ has dynamical degree one (see Theorem \ref{thm:Xie} for details).  Thus any counterexamples to either Conjecture \ref{conj:strong} or \ref{conj:iff} must have dimension at least three.  

Another important class of maps for which we can show Conjecture \ref{conj:strong} holds is for automorphisms $\phi$ that lie in the connected component ${\rm Aut}_0(X)$ of the identity of the automorphism group of an irreducible complex algebraic variety $X$.  In this case we, we consider the Zariski closure, $H$, of the subgroup of ${\rm Aut}_0(X)$ generated by $\sigma$ and apply Chevalley's theorem on constructible sets \cite[Theorem 3.16]{Harris} to deduce that if there is some point $x$ whose orbit under $H$ is Zariski dense, then the $H$-orbit contains a dense open subset $U$ of $X$. Thus, every point in $U$ will have dense orbit under $\phi$ and so we see condition (ii) holds unless no point in $X$ has a Zariski dense orbit, which in turn implies (i) holds.
     
Automorphisms $\phi$ lying in the connected component of the automorphism group of $X$, as above, all have dynamical degree one and one can regard rational self-maps of dynamical degree one as being a natural generalization of this important class of self-maps.  The two classes mentioned above (rational self-maps of surfaces of dynamical degree one and automorphisms in a connected algebraic group) give underpinning to Conjecture \ref{conj:strong}.

In Section~\ref{sec:future}, we connect our results with results concerning the representation theory of noncommutative algebras. In particular, we consider the class of algebras called \emph{twisted homogeneous coordinate rings}, which are constructed from a projective variety $X$, an automorphism $\sigma$ of $X$, and an ample invertible sheaf $\mathcal{L}$. Here it is known that the noetherian property for these algebras holds precisely when $\sigma$ has dynamical degree one, and the biregular case of Conjectures \ref{conj:strong} and \ref{conj:iff} for complex projective varieties is equivalent to existing conjectures about the representation theory for this class of algebras.


\subsection{Our results}

We prove Conjectures~\ref{conj:strong}~and~\ref{conj:iff} (even in slightly stronger forms) for regular self-maps of semiabelian varieties. We recall that a semiabelian variety (over an algebraically closed field) is an extension of an abelian variety by a power of the multiplicative group. Also, in order to define the dynamical degree for a self-map of a semiabelian variety $X$, one could consider a suitable compactification of $X$; however, as explained in Section~\ref{subsec:strategy} (see Proposition~\ref{prop:fact}), there is a simple characterization of dominant regular self-maps of dynamical degree $1$ for semiabelian varieties. 
\begin{theorem}
\label{thm:main}
Let $X$ be a semiabelian variety defined over an algebraically closed field $K$ of characteristic $0$, and let $\phi:X\longrightarrow X$ be an automorphism of dynamical degree $1$. Then exactly one of the following two statements must hold:
\begin{enumerate}
\item[(1)] either there exists a non-constant rational function $f:X\dashrightarrow \P^1$ such that $f\circ \phi=f$; 
\item[(2)] or there exists no proper $\phi$-invariant subvariety (equivalently, there exists no proper irreducible subvariety $Y\subset X$ which is periodic under the action of $\phi$, i.e., $\phi^\ell(Y)= Y$ for some $\ell\in\N$).
\end{enumerate}
\end{theorem}

\begin{remark}
\label{rem:equivalent}
The equivalence of the two statements from conclusion~(2) in Theorem~\ref{thm:main} is immediate since given an automorphism $\phi$ of some variety $X$, for any proper invariant subvariety $V$, its irreducible components must be periodic under the action of $\phi$. A similar argument applies also for the equivalent statement appearing in our next result.
\end{remark}

On the other hand, if the dynamical degree of a regular self-map $\phi$ on a semiabelian variety $X$ is greater than $1$, then we can prove that \emph{always} (regardless whether $\phi$ preserves a non-constant rational function or not) the union of the proper $\phi$-invariant subvarieties of $X$ is Zariski dense.
\begin{theorem}
\label{thm:main_2}
Let $X$ be a semiabelian variety defined over an algebraically closed field $K$ of characteristic $0$, and let $\phi:X\longrightarrow X$ be a dominant regular self-map of dynamical degree larger than $1$.   Then the union of all $\phi$-invariant proper subvarieties of $X$ is Zariski dense; equivalently, the union of all irreducible proper $\phi$-periodic subvarieties $Y\subset X$ (i.e., $\phi^\ell(Y)= Y$ for some $\ell\in\N$) is Zariski dense in $X$.
\end{theorem}

We prove Theorems~\ref{thm:main}~and~\ref{thm:main_2} in Sections~\ref{sec:proof}~and~\ref{sec:proof_2}, respectively. We also discuss further directions for studying Conjectures~\ref{conj:strong}~and~\ref{conj:iff} in Section~\ref{sec:future}.


\subsection{Strategy for our proof}
\label{subsec:strategy}

In our proofs for both Theorems~\ref{thm:main}~and~\ref{thm:main_2}, we will employ the following characterization of regular dominant self-maps of semiabelian varieties of dynamical degree equal to $1$. First, we recall that (regardless of characteristic of the field of definition, as proven in \cite{Iitaka}), the regular self-maps of a semiabelian variety are compositions of translations with algebraic group endomorphisms. Furthermore, for any group endomorphism $\Psi$ of a semiabelian variety, there exists a monic polynomial $P\in\Z[x]$ such that $P(\Psi)=0$; for more details, we refer the reader to \cite[Section~2.1]{CGSZ}.

\begin{proposition}
\label{prop:fact}
Let $X$ be a semiabelian variety defined over a field of characteristic $0$ and let $\phi:=T\circ \Psi$ be a regular dominant self-map of $X$, where $T:X\longrightarrow X$ is a translation, while $\Psi$ is an algebraic group endomorphism of $X$. Let $P(x)$ be the minimal polynomial for $\Psi$. Then the dynamical degree of $\phi$ equals $1$ if and only if each root of $P(x)$ is a root of unity.
\end{proposition}

\begin{proof}
The proof of this fact is essentially covered in \cite{Matsuzawa-Sano}. First of all, the dynamical degree of $\phi$ equals the dynamical degree of $\Psi$ (since each iterate $\phi^n$ of $\phi$ is a composition of $\Psi^n$ with a suitable translation). Second, $\lambda_1(\Psi)=1$ if and only if the spectral radius of $\Psi^*|_{H^1(X)}$ is equal to $1$ and so, all roots of the polynomial $P$ must have absolute value equal to $1$ (for more details, see \cite{Matsuzawa-Sano}). Then a classical theorem of Kronecker regarding algebraic numbers whose Galois conjugates all have absolute value equal to $1$ yields that all roots of $P(x)$ must be roots of unity, as desired.
\end{proof}

Assume now that the dynamical degree of $\phi:=T\circ \Psi$ equals $1$ (as in Theorem~\ref{thm:main}). Then we get that there exist positive integers $\ell$ and $m$ such that 
\begin{equation}
\label{eq:cyclotomic}
\left(\Psi^\ell - {\rm id}_X\right)^m=0.
\end{equation}
Since the conclusion in Theorem~\ref{thm:main} is unaltered if we replace our self-map $\phi$ by an iterate of it (which can be seen by looking at the irreducible periodic subvarieties $Y\subset X$, as in  Remark~\ref{rem:equivalent}), then replacing $\phi$ by $\phi^\ell$ (see Equation~\eqref{eq:cyclotomic}), we may assume that $\Psi$ is an unipotent algebraic group endomorphism. Then employing \cite[Theorem~7.2]{Zinovy} (along with \cite[Theorem~3.1]{P-R}) allows us to finish the proof of Theorem~\ref{thm:main}; in the language of \cite{Zinovy}, the automorphism $\phi$ is \emph{wild} (see Section~\ref{subsec:wild}) and so, it does not have proper $\phi$-invariant subvarieties.  We also note that one could  obtain the desired conclusion from Theorem~\ref{thm:main} by using alternatively  more combinatorial arguments akin to the ones employed in the proof from \cite{G-Sina-20}.

Finally, in order to prove Theorem~\ref{thm:main_2} (whose conclusion is once again unchanged if one replaces $\phi=T\circ \Psi$ by a suitable iterate of it), we analyze the action of $\Psi$ on $X$ according to the roots of its minimal polynomial $P(x)$; for this part our arguments are somewhat similar to the ones employed in \cite{G-S, G-Matt}.


\section{Proof of Theorem~\ref{thm:main}}
\label{sec:proof}


\subsection{General setup for our proof}
\label{subsec:general}

In this section we prove Theorem~\ref{thm:main}; so, we work under the hypotheses of Theorem~\ref{thm:main} for an automorphism $\phi$ of dynamical degree $1$ of a semiabelian variety $X$ defined over an algebraically closed field $K$ of characteristic $0$. Also, we have that the automorphism $\phi$ is a composition of a translation $T$ (i.e., for each $x\in X$, we have $T(x)=x+\gamma$ for some given point $\gamma\in X(K)$) with an algebraic group automorphism $\Psi$. Furthermore, as explained in Section~\ref{subsec:strategy} (note that replacing $\phi$ by an iterate does not change the set of $\phi$-invariant subvarieties), at the expense of replacing $\phi$ by a suitable iterate, we may assume $\Psi$ is unipotent, i.e.,
\begin{equation}
\label{eq:unipotent}
\left(\Psi - {\rm id}_X\right)^m=0,
\end{equation}
for some positive integer $m$.

Also, from now on, we assume $\phi$ does not preserve a non-constant fibration, i.e., condition~(1) from Theorem~\ref{thm:main} does not hold. Then we will prove that there are no proper $\phi$-invariant subvarieties.


\subsection{Analyzing the unipotent part of the automorphism}
\label{subsec:nilpotent}

We continue with our notation and convention for our automorphism $\phi=T\circ \Psi$ of the semiabelian variety $X$:
\begin{itemize}
\item[(a)] $\phi$ preserves no non-constant fibration;
\item[(b)] $T$ is a translation by a point $\gamma\in X(K)$; and
\item[(c)] $\Phi$ is a unipotent algebraic group automorphism, i.e., the map $\Phi_0:=\Phi-{\rm Id}_X$ is a nilpotent algebraic group endomorphism for $X$.
\end{itemize}

The following technical result (which is inspired by \cite[Theorem~7.2]{Zinovy})  will be crucially employed in Section~\ref{subsec:wild}.

\begin{proposition}
\label{prop:nilpotent}
We let $\bar{X}:=X/\Phi_0(X)$ and for each point $\alpha\in X$, we denote by $\bar{\alpha}$ its image under the natural projection map $\pi:X\longrightarrow \bar{X}$. Under the notation and assumptions from conditions~(a)--(c) above, we must have that the cyclic group generated by $\bar{\gamma}$ is Zariski dense in $\bar{X}$.
\end{proposition}

\begin{proof}
We argue by contradiction and therefore, assume there exists a proper algebraic subgroup $\bar{Y}\subset \bar{X}$ containing $\bar{\gamma}$. Then we let $Y:=\pi^{-1}\left(\bar{Y}\right)$, which is a proper algebraic subgroup of $X$. We claim that the projection map $g:X\longrightarrow X/Y$ is left invariant by our automorphism $\phi$. Indeed, for any point $x\in X$, we have that 
$$\phi(x)-x=\Phi_0(x)+\gamma\in Y,$$
and thus, $g\circ \phi = g$, as claimed. Since $g$ is not the trivial map (because $Y$ is a proper algebraic subgroup of $X$), we contradict condition~(a) above; hence $\bar{\gamma}$ must generate indeed $\bar{X}$. This concludes our proof of Proposition~\ref{prop:nilpotent}. 
\end{proof}


\subsection{Wild automorphisms}
\label{subsec:wild}

The following notion was studied in \cite{Zinovy}.
\begin{definition}
\label{def:wild}
An automorphism $\phi$ of some variety $X$ is called wild if the orbit of each point in $X$ is Zariski dense.
\end{definition}

It is immediate to see that if an automorphism $\phi:X\longrightarrow X$ is wild, then there are no proper $\phi$-invariant subvarieties. So, we are left to show that our automorphism $\phi=T\circ \Psi$ of the semiabelian variety $X$ is wild. However, since $\phi$ satisfies the conclusion of Proposition~\ref{prop:nilpotent} (note that we are working under the assumption~(a) above saying that $\phi$ leaves invariant no non-constant rational function), then \cite[Theorem~7.2]{Zinovy} delivers the desired conclusion that the automorphism $\phi$ must be wild. 

Now, strictly speaking, \cite[Theorem~7.2]{Zinovy} is proven under the assumption that $X$ is an abelian variety. However, its proof goes verbatim in the case $X$ is semiabelian since the only part where the authors of \cite{Zinovy} employed the assumption about $X$ being abelian was to infer that any irreducible $\phi$-invariant subvariety of $X$ must be a translate of a (connected) algebraic subgroup. For this last result, indeed they used the fact that $X$ was abelian, as in their proof from \cite[Corollary~4.3]{Zinovy}. However, we can replace the use of \cite[Corollary~4.3]{Zinovy} with the use of \cite[Theorem~3.1]{P-R}, which would still guarantee that also in the semiabelian case, the irreducible $\phi$-invariant subvarieties must be cosets of algebraic subgroups. Indeed, the assumption~(a) above that $\phi$ admits no non-constant invariant fibration means that $\phi$ does not induce a finite order automorphism of a nontrivial quotient of $X$ and therefore, according to \cite[Theorem~3.1]{P-R}, each $\phi$-invariant irreducible subvariety $Z$ of $X$ must have trivial quotient through its stabilizer $W$; hence $Z$ would be a coset of the algebraic subgroup $W$, as desired. 

So, indeed, $\phi:X\longrightarrow X$ is a wild automorphism; therefore, there  are no proper $\phi$-invariant subvarieties. This concludes our proof of Theorem~\ref{thm:main}.


\section{Proof of Theorem~\ref{thm:main_2}}
\label{sec:proof_2}


\subsection{Generalities}
\label{subsec:generalities_2}

We work under the hypotheses from Theorem~\ref{thm:main_2}; in particular, we let $\phi=T\circ \Psi$, where $T$ is a translation on the semiabelian variety $X$, while $\Psi$ is an algebraic group endomorphism.

In our proof of Theorem~\ref{thm:main_2} we may replace $\phi$ by its conjugate $T_\alpha\circ \phi\circ T_{-\alpha}$ (where $T_c$ always represents the translation-by-$c$ map for any given point $c\in X(K)$) since this would not affect the dynamical degree of our map, nor the conclusion that the union of all invariant subvarieties is Zariski dense; note that $Z$ is $\phi$-invariant if and only if $Z+\alpha$ is invariant under $T_\alpha\circ \phi\circ T_{-\alpha}$.


\subsection{The minimal polynomial}

We let $P(x)\in\Z[x]$ be the (monic) minimal polynomial for $\Psi$. At the expense of replacing $\phi$ by a suitable iterate (which, in particular, leads to replacing $\Psi$ by the corresponding iterate), we may assume that each root of $P(x)$ is either equal to $1$, or it is not a root of unity (nor equal to $0$, since $\Psi$ must be an isogeny because $\phi$ is a dominant map). Then we write
$$P(x)=(x-1)^r\cdot Q(x),$$
for some non-negative integer $r$ (which is the order of the root $1$ in $P(x)$)  and some (monic) polynomial $Q(x)\in\Z[x]$. Now, since we assumed that $\phi$ (and therefore $\Psi$) has dynamical degree larger than $1$, then it means that $P(x)$ has at least one root which is not a root of unity and so, $Q(x)$ must be a non-constant polynomial (whose roots are not roots of unity, by our assumption that all roots of unity appearing among the roots of the polynomial $P(x)$ must equal $1$).


\subsection{Decomposing the action of our map}

We consider the following connected algebraic subgroups of $X$, defined as follows: $X_2:=\left(\Psi-{\rm Id}_X\right)^r(X)$ and also, let $X_1:=Q(\Psi)(X)$. We note that if $r=0$, then $X_1$ is the trivial semiabelian variety. On the other hand, since $P(x)\ne (x-1)^r$ (because the dynamical degree of $\phi$ and thus of $\Psi$ is not equal to $1$), we know that 
\begin{equation}
\label{eq:nontrivial_semiabelian}
X_2\text{ is a nontrivial semiabelian variety.}
\end{equation}

Since the polynomials $(x-1)^r$ and $Q(x)$ are coprime (and their product kills the endomorphism $\Psi$), then arguing as in \cite[Lemma~6.1]{G-S} (see also the explanation around \cite[Equation~(4.0.2)]{G-Matt}), we have that 
\begin{equation}
\label{eq:disjoint}
X_1+X_2=X\text{ and }X_1\cap X_2\text{ is finite.}
\end{equation}
So, letting our translation map $T:X\longrightarrow X$ correspond to the point $\gamma\in X(K)$, then we can find $\gamma_i\in X_i(K)$ (for $i=1,2$) such that $\gamma=\gamma_1+\gamma_2$.  Also, $\Psi$ induces dominant algebraic group endomorphisms $\Psi_i:=\Psi|_{X_i}$ for $i=1,2$. Furthermore, the minimal polynomial of $\Psi_1$ (as an endomorphism of $X_1$) is $(x-1)^r$, while the minimal polynomial for $\Psi_2$ (as an endomorphism of $X_2$) is $Q(x)$.  For each $i=1,2$, we let $\phi_i:X_i\longrightarrow X_i$ be given by the composition of the translation-by-$\gamma_i$ with the group endomorphism $\Psi_i$. Finally, we have that for each point $x\in X$ written as $x=x_1+x_2$ for $x_i\in X_i$ (see equation~\eqref{eq:disjoint}), then we have
\begin{equation}
\label{eq:disjoint_2}
\phi(x)=\phi_1(x_1)+\phi_2(x_2).
\end{equation}


\subsection{Conjugating one of the maps to a group endomorphism}

Since the minimal polynomial of $\Psi_2:X_2\longrightarrow X_2$ does not have roots equal to $1$ (actually, not even roots of unity), the algebraic group endomorphism $\Psi_2-{\rm Id}_{X_2}$ (of $X_2$) must be dominant, and so there exists $\beta_2\in X_2(K)$ such that $\left(\Psi_2-{\rm Id}_{X_2}\right)(\beta_2)=\gamma_2$. Then conjugating $\phi_2$ by the translation $T_{\beta_2}$ given by $\beta_2$ (i.e., replacing $\phi_2$ by $T_{\beta_2}\circ \phi_2\circ T_{-\beta_2}$) we obtain the group endomorphism $\Psi_2$. 

So, at the expanse of conjugating $\phi$ by the translation-by-$\beta_2$ map on $X$ (note that $\beta_2\in X_2\subseteq X_1$), we may assume that the dominant regular map $\phi_2:X_2\longrightarrow X_2$ is an algebraic group endomorphism (also note, as explained in Section~\ref{subsec:generalities_2}, that we can always replace our map with a conjugate of it by a translation map).


\subsection{Periodic points for an algebraic group endomorphism}

The following easy fact will be crucial in the conclusion of our proof from Section~\ref{subsec:conclusion_2}.

\begin{proposition}
\label{prop:easy_fact}
Let $Z$ be a semiabelian variety defined over a field of characteristic $0$ and let $\Phi$ be a dominant algebraic group endomorphism of $Z$. Then the set of periodic points of $Z$ is Zariski dense.
\end{proposition}

\begin{proof}
Indeed, each torsion point of $Z$ of order coprime with $\#\ker(\Phi)$ must be periodic under the action of $\Phi$; hence, there exists a Zariski dense set of $\Phi$-periodic points.
\end{proof}


\subsection{Conclusion of our proof}
\label{subsec:conclusion_2}

We let $\tilde{X}:=X_1\oplus X_2$ and let $\tilde{\phi}$ be the dominant map on $\tilde{X}$ given by $(\phi_1,\phi_2)$. Then we let the isogeny $\iota:X_1\oplus X_2\longrightarrow X$ (see also Equation~\ref{eq:disjoint}) given by 
$$\iota(x_1,x_2)=x_1+x_2.$$
It is immediate to check (see equation~\eqref{eq:disjoint_2}) that we have a commutative diagram, i.e., 
\begin{equation}
\label{eq:commutative_diagram}
\phi\circ \iota = \iota\circ \tilde{\phi}.
\end{equation} 
Therefore, Equation~\eqref{eq:commutative_diagram} yields that for any proper $\tilde{\phi}$-invariant subvariety $\tilde{Z}\subset \tilde{X}$, $\iota(\tilde{Z})$ is a proper $\phi$-invariant subvariety of $X$. 

On the other hand, Proposition~\ref{prop:easy_fact} yields that there exists a Zariski dense set of $\tilde{\phi}$-invariant subvarieties of $\tilde{X}$ of the form $X_1\times S$, where $S\subset X_2$ is a finite set of periodic points under the action of the endomorphism $\phi_2$; note that $X_2$ is positive dimensional (see equation~\eqref{eq:nontrivial_semiabelian}) and so, indeed, $X_1\times S$ is a proper subvariety of $\tilde{X}$. Therefore, the set of proper $\phi$-invariant subvarieties of $X$ is Zariski dense.

This concludes our proof of Theorem~\ref{thm:main_2}.
\subsection{Results for surfaces}
We point out that Conjectures \ref{conj:strong} and \ref{conj:iff} were already known for surfaces, due to work of Cantat \cite{Can}, Diller and Favre \cite{DF}, and Xie \cite{Xie1}.  We give an argument for the sake of completeness, although we stress that these results are well-known to experts.
\begin{theorem} \label{thm:Xie}
Conjectures \ref{conj:strong} and \ref{conj:iff} hold whenever $X$ is a surface.
\end{theorem}
\begin{proof}
It suffices to show Conjecture \ref{conj:iff}  holds. By \cite[Theorem 1.1]{Xie1} if $\phi$ does not preserve a non-constant rational fibration then either the dynamical degree of $\phi$ is one or the union of the periodic points is Zariski dense, and exactly one of these alternatives hold.  Further, a result of Cantat \cite{Can} shows that if $\phi$ does not preserve a non-constant rational fibration then there are at most finitely many $\phi$-periodic curves, and so in the case that the dynamical degree of $\phi$ is one, there is a maximal invariant proper subvariety of $X$, unless $\phi$ preserves a non-constant fibration.   
\end{proof}


\section{Connections with irreducible representations of algebras}
\label{sec:future}
In this section, we explore connections between Conjecture \ref{conj:strong} and representation theoretic questions concerning a class of associative algebras constructed from geometric data.  Much of this is connected with earlier work from \cite{BRS}. 

A classical construction in algebraic geometry is the homogeneous coordinate ring $R$ for a projective variety $X$ over an algebraically closed field $k$.  This ring $R$ is graded by the natural numbers and has the property that one can naturally identify ${\rm Proj}(R)$ with the projective scheme $X$.  In general, the homogeneous coordinate ring is not uniquely defined and depends instead upon some embedding of $X$ into $\mathbb{P}^n$.  More precisely, one fixes an ample invertible sheaf $\mathcal{L}$ and one forms the ring
\begin{equation} R:=\bigoplus_{n\ge 0} H^0(X,\mathcal{L}^{\otimes n}).\end{equation}  

In the early `90s it was observed that certain questions motivated by work in mathematical physics could be approached by considering a twisted version of the above construction \cite{AV, ATV}.  In this setting, one again has a projective variety $X$ and ample invertible sheaf $\mathcal{L}$ but, in addition to this data, one fixes an automorphism $\sigma$ of $X$, which is used to ``twist'' the multiplication of the ring $R$.  Here we take 
$$\mathcal{L}_n:=\mathcal{L}\otimes \sigma^*(\mathcal{L})\otimes \cdots \otimes (\sigma^{n-1})^*(\mathcal{L})$$ for $n\ge 0$, where $(\sigma^i)^*(\mathcal{L})$ is the pullback of $\mathcal{L}$ along $\sigma^i$ and where we take 
$\mathcal{L}_0=\mathcal{O}_X$ and we define
$$B(X,\mathcal{L},\sigma):= \bigoplus_{n\ge 0} H^0(X,\mathcal{L}_n),$$ and we endow this vector space with bilinear multiplication
$$\star: H^0(X,\mathcal{L}_n)\times H^0(X,\mathcal{L}_m)\to H^0(X,\mathcal{L}_{n+m})^*(\mathcal{L})$$ given by $f\star g = f\cdot (\sigma^n)^*(g)$ for
$f\in H^0(X,\mathcal{L}_n)$ and $g\in H^0(X,\mathcal{L}_m)$, where $\cdot$ is the usual bilinear map $H^0(X,\mathcal{E})\times H^0(X,\mathcal{F}) \to H^0(X,\mathcal{E}\otimes \mathcal{F})$ for invertible sheaves $\mathcal{E}$ and $\mathcal{F}$.

Then under this new multiplication, $B(X,\mathcal{L},\sigma)$ becomes an associative algebra, which is called the \emph{twisted homogeneous coordinate ring of} $X$ (with respect to $\sigma$ and $\mathcal{L}$). 

There is a striking dichotomy that arises when one looks at the behavior of these algebras in terms of the automorphism $\sigma$: when $\sigma$ has dynamical degree one, the twisted homogeneous coordinate ring is noetherian and has finite Gelfand-Kirillov dimension (a noncommutative analogue of Krull dimension); and when $\sigma$ has dynamical degree strictly larger than one, the twisted homogeneous coordinate ring is non-noetherian and has exponential growth \cite{Kee}.  The algebraic implication of this dichotomy is that one expects the representation theory of twisted homogeneous coordinate rings to be much nicer than in the case that the automorphism has dynamical degree one.  

One of the most important methods in studying an algebra $A$ is to understand the underlying structure of its irreducible representations (that is, the simple left $A$-modules). In practice, it is often very difficult to obtain an explicit description of these representations and so one often settles instead for a coarser understanding by characterizing the annihilators of simple modules.

These annihilator ideals of simple modules of an algebra are called the \emph{primitive} ideals and they form a distinguished subset of the prime spectrum of the algebra. Due to their structure theoretic importance, their study enjoys a long history.  The first major achievement in this direction was the work of Dixmier \cite{Dix} and Moeglin \cite{Moe}, which shows that the primitive ideals of an enveloping algebra of a finite-dimensional complex Lie algebra can be characterized in both topological and algebraic terms.
\begin{theorem} \em{(}Dixmier-Moeglin \cite{Dix,Moe}{\rm )} Let $L$ be a finite-dimensional complex Lie algebra and let $U(L)$ be its enveloping algebra.  Then for a prime ideal $P$ of $U(L)$ the following are equivalent:
\begin{enumerate}
\item $P$ is primitive;
\item $\{P\}$ is an open subset of its closure in ${\rm Spec}(U(L))$, where we endow the prime spectrum with the Zariski topology;
\item $U(L)/P$ has a simple Artinian ring of fractions whose centre is the base field $\mathbb{C}$.
\end{enumerate}
\end{theorem}

There is a theory of noncommutative localization due to Goldie \cite[Chapt. 2]{MR}, which gives that if $P$ is a prime ideal of a noetherian $k$-algebra $A$ then $A/P$ has a ring of fractions, which we denote ${\rm Frac}(A/P)$, and which is a generalization of the field of fractions construction for commutative integral domains.  This ring of quotients is of course not a field in general, but it is simple Artinian and hence isomorphic to a matrix ring over a division $k$-algebra.  In particular, its centre is a field extension of $k$.  The third condition in the list of equivalent conditions given by Dixmier and Moeglin then says that primitivity of $P$ is in some sense equivalent to $U(L)/P$ being as ``noncommutative as possible'' in the sense of having a ring of fractions whose centre is as small as possible.

In general, given a noetherian algebra $A$ over an algebraically closed field $k$, we say that a prime ideal $P$ is \emph{rational} if ${\rm Frac}(A/P)$ has centre $k$; and we say that $P$ is \emph{locally closed} if $\{P\}$ is an open subset of its closure in ${\rm Spec}(A)$.  In honour of the work of Dixmier and Moeglin, a $k$-algebra $A$ for which the properties of being primitive, locally closed, and rational are equivalent for all primes $P\in {\rm Spec}(A)$ is said to satisfy the \emph{Dixmier-Moeglin equivalence}.  

The Dixmier-Moeglin equivalence is now known to hold for a large class of noetherian algebras, including many natural classes of quantum algebras and Hopf algebras \cite{GL, BG}.  In general, the Dixmier-Moeglin equivalence holds for most known examples of noetherian algebras of finite Gelfand-Kirillov dimension; there are exceptions, but they are somewhat rare and tend to be difficult to construct (see, for example, \cite{BLLM}).  

It has been conjectured that the Dixmier-Moeglin equivalence holds for noetherian twisted homogeneous coordinate rings $B(X,\mathcal{L},\sigma)$ \cite[Conjecture 8.5]{BRS}.  The noetherian property is equivalent to the automorphism $\sigma$ having dynamical degree one \cite{Kee}; this conjecture has been established when ${\rm dim}(X)\le 2$.    

In this setting, one can give a purely geometric characterization of the properties of being primitive, rational, and locally closed in terms of $\sigma$-periodic irreducible subvarieties of $X$.

\begin{proposition} \em{(}\cite[Theorem 8.1(1)]{BRS}\rm{)} Let $X$ be a complex irreducible projective variety, let $\mathcal{L}$ be an ample invertible sheaf, and let $\sigma\in {\rm Aut}_{\mathbb{C}}(X)$.  Then $B(X,\mathcal{L},\sigma)$ satisfies the Dixmier-Moeglin equivalence if $\sigma$ has dynamical degree one and for every subvariety $\sigma$-invariant subvariety $Y$ of $X$, the union of the $\sigma$-invariant proper subvarieties of $Y$ is a Zariski closed subset of $Y$.  \label{prop:DM}
\end{proposition}
In particular, applying Theorems \ref{thm:main} and \ref{thm:main_2} and using the criterion in Proposition \ref{prop:DM}, we can deduce that the Dixmier-Moeglin equivalence holds for noetherian twisted homogeneous coordinate rings of abelian varieties.  We prove a more general result for split semiabelian varieties. We recall that a semiabelian variety over an algebraically closed field is \emph{split} if is isogenous to a direct product of an abelian variety and a power of the multiplicative group.    
\begin{proposition} 
\label{prop:4.3}
Let $X$ be a split semiabelian variety over an algebraically closed field of characteristic zero, let $\Phi$ be an algebraic group  automorphism of $X$, let $a\in X$, and let $\sigma: X\to X$ be the map $\sigma(x)=\Phi(x)+a$. Then:
\begin{enumerate}
\item if $\sigma$ has dynamical degree $>1$ then there is a $\sigma$-invariant subvariety $Y$ of $X$ with the property that the union of the $\sigma$-invariant proper subvarieties of $Y$ is a Zariski dense, proper subset of $Y$;
\item if $\sigma$ has dynamical degree $1$ then every $\sigma$-invariant subvariety $Y$ of $X$ has the property that the union of the $\sigma$-invariant proper subvarieties of $Y$ is a Zariski closed subset of $Y$. 
\end{enumerate}\label{prop:PR}
\end{proposition}
\begin{proof} We prove this by induction on the dimension of $X$.  When ${\rm dim}(X)=0$, there is nothing to prove. Thus we assume that (1) and (2) hold whenever ${\rm dim}(X)<d$ with $d\ge 1$ and consider the case when ${\rm dim}(X)=d$.  

By Theorems \ref{thm:main} and \ref{thm:main_2}, we obtain both (1) and (2) if $\sigma$ does not preserve a non-constant fibration.  Indeed, Theorem~\ref{thm:main} yields that if $\phi$ has dynamical degree $1$, then there exist no proper $\sigma$-invariant subvarieties of $X$ and so, conclusion~(2) holds trivialy. On the other hand, if the dynamical degree of $\phi$ is larger than $1$ then the assumption that $\sigma$ does not preserve a non-constant fibration yields (according to \cite{G-Matt}) that there exists a point $x\in X$ with a Zariski dense orbit; therefore, $x$ would not be contained in a proper $\sigma$-invariant subvariety $Z\subset X$. However, Theorem~\ref{thm:main_2} yields that the union of all proper $\sigma$-invariant subvarieties of $X$ would still be Zariski dense in $X$; thus, conclusion~(1) in Proposition~\ref{prop:4.3} holds for $X$ itself.

Therefore, from now on, we may assume that $\sigma$ preserves a non-constant fibration.  

Since a variety is $\sigma$-periodic if and only if it is $\sigma^r$-periodic, we may replace $\sigma$ by $\sigma^r$. So, letting $P(x)$ be the monic, minimal polynomial for the algebraic group automorphism $\Phi$, at the expense of replacing $\Phi$ by $\Phi^r$ (and thus, replacing $\sigma$ by $\sigma^r$), we may assume each root of $P(x)$ is either equal to $1$, or it is not a root of unity.

By \cite[Theorem 1.2]{G-Sina}, since $\sigma$ preserves a non-constant fibration, there exists a non-constant group endomorphism $\Psi:X \to X$  such that $\Psi \circ (\Phi - {\rm Id})$ is $0$ in the endomorphism ring of $X$, and furthermore $\Psi\circ \sigma = \Psi$ (i.e., $a\in \ker(\Psi)$).


Let $E$ denote the connected component (of the identity) of the kernel of $\Psi$; since $\Psi$ is non-constant, then $E\ne X$ and so, $E$ is a split semiabelian subvariety of dimension $<d$. We let $\pi$ be the quotient homomorphism $\pi: X\to X/E$, then $\pi\circ \sigma =\pi$.  Then $\sigma|_{E}$ is an automorphism of $E$ and if the dynamical degree of $\sigma$ is strictly larger than one, then the dynamical degree of $\sigma|_E$ is also strictly larger than one by Proposition \ref{prop:fact} and the fact that the minimal polynomial of $\Psi|_{E}$ has all the roots of the minimal polynomial of $\Psi$, except, possibly, the root equal to $1$.  Hence by the induction hypothesis we obtain conclusion~(1) in Proposition~\ref{prop:4.3}.

Now, we are left to prove conclusion~(2) in Proposition~\ref{prop:4.3}. So,  the dynamical degree of $\sigma$ is one and suppose towards a contradiction that there is some $\sigma$-invariant subvariety $Y$ of $X$ such that the union of the proper $\sigma$-invariant subvarieties of $Y$ is not a  Zariski closed subset of $Y$.  Then since $\sigma$ permutes the irreducible components of $Y$, after replacing $\sigma$ by a suitable iterate and taking a suitable irreducible component of $Y$, we may assume without loss of generality that $Y$ is irreducible.  

Now, if  $\pi(Y)$ is a point, then (at the expense of replacing $\sigma$ by a conjugate of it by a suitable translation), we may assume that $Y\subseteq E$.  By Proposition \ref{prop:fact}, $\sigma|_{E}$ has also dynamical degree one and so by the induction hypothesis the union of the proper invariant subvarieties of $Y$ is Zariski closed, as desired in part~(2) of Proposition~\ref{prop:4.3}.

Therefore, the remaining case is when $\pi(Y)$ is a positive dimensional subvariety of $X/E$. But then $\sigma|_Y:Y\longrightarrow Y$ preserves a non-constant fibration and so, the union of the proper $\sigma$-invariant subvarieties of $Y$ equals the entire $Y$, as desired once again in part~(2) of Proposition~\ref{prop:4.3}. 

This concludes our proof of Proposition~\ref{prop:4.3}.
\end{proof} 

Proposition~\ref{prop:4.3} yields the following corollary.

\begin{corollary} Let $X$ be a complex abelian variety, let $\mathcal{L}$ be an ample invertible sheaf, and let $\sigma\in {\rm Aut}_{\mathbb{C}}(X)$. If $B(X,\mathcal{L},\sigma)$ is noetherian then it satisfies the Dixmier-Moeglin equivalence.  
\end{corollary}
\begin{proof}
By \cite[Theorem 1.3]{Kee}, $B(X,\mathcal{L},\sigma)$ is noetherian if and only if $\sigma$ has dynamical degree one.  The result now follows from Propositions \ref{prop:DM} and \ref{prop:PR}.  \end{proof}
We note that in an earlier paper \cite{Advances}, we considered other dynamical questions for endomorphisms of semiabelian varieties and used a similar translation of dynamical results to obtain topological and algebraic characterizations of primitive ideals in skew polynomial extensions of $\mathbb{C}[x_1^{\pm 1},\ldots ,x_d^{\pm 1}]$ (see \cite[Theorem 1.1]{Advances}).  This class of algebras shares some commonalities with twisted homogeneous coordinate rings of abelian varieties in that they are both constructed from a semiabelian variety along with an automorphism of this variety, although in the latter case the ambient variety is projective while in the former case it is affine.
\section*{Acknowledgments} We thank Fei Hu for many helpful comments.

\end{document}